\documentclass{amsart}
\usepackage{amssymb}
\usepackage{amsfonts}
\usepackage{amsmath}
\usepackage{graphicx}
\usepackage{shadow}
\usepackage{color}
\usepackage[all]{xy}
\usepackage[pagebackref]{hyperref}

\newtheorem{thm}{Theorem}[section]
\newtheorem{theorem}[thm]{Theorem}
\newtheorem{corollary}[thm]{Corollary}
\newtheorem{lemma}[thm]{Lemma}
\newtheorem{proposition}[thm]{Proposition}

\theoremstyle{definition}
\newtheorem{definition}[thm]{Definition}
\newtheorem{example}[thm]{Example}

\theoremstyle{remark}
\newtheorem{remark}[thm]{Remark}

\DeclareMathOperator{\m}{\frak{m}}
\DeclareMathOperator{\p}{\frak{p}}
\DeclareMathOperator{\PP}{\mathcal{P}}

\DeclareMathOperator{\Hom}{Hom}
\DeclareMathOperator{\Ext}{Ext}

\DeclareMathOperator{\Tor}{Tor}

\def\gldim{{\rm gl.dim}}

\def\w-wgldim{{\rm w-w.gl.dim}}
\def\wgldim{{\rm w.gl.dim}}

\def\s-wgldim{u$-$ S$-w.gl.dim$}

\everymath{\displaystyle\everymath{}}
\begin{document}
	
	\title[ON $w$-COPURE PROJECTIVE MODULES]{ON $w$-COPURE PROJECTIVE MODULES}
\author[R.A.K. Assaad]{Refat Abdelmawla Khaled Assaad}
\address{Department of Mathematics, Faculty of Science,  Moulay Ismail University - Meknes, Box 11201, Zitoune, Morocco}	
\email{refat90@hotmail.com}
	
\author[M. Tamekkante]{Mohammed Tamekkante}
\address{Department of Mathematics, Faculty of Science,  Moulay Ismail University - Meknes, Box 11201, Zitoune, Morocco}
\email{tamekkante@yahoo.fr}

\author[L. Mao]{Lixin Mao}
\address{School of Mathematics and Physics, Nanjing Institute of Technology, Nanjing, China}
\email{maolx2@hotmail.com}


	\begin{abstract}
Let $R$ be a commutative ring. An $R$-module $M$ is said to be
$w$-split if $\Ext_{R}^1(M,N)$ is a GV-torsion $R$-module for all $R$-modules $N$. It is known that every projective module is $w$-split, but the converse is not true in general. In this paper, we study the w-split dimension of a flat module. To do so, we introduce and study the so-called $w$-copure (resp., strongly $w$-copure) projective modules which is in some way a generalization of the notion of copure (resp., strongly copure) projective modules. An $R$-module $M$ is said to be $w$-copure projective (resp., strongly $w$-copure projective) if $\Ext_{R}^1(M,N)$ (resp., $\Ext_{R}^n(M,N)$) is a GV-torsion $R$-module for all flat $R$-modules $N$ and any $n\geq1$.
\end{abstract}
	\maketitle

	\section{\bf Introduction}
Throughout, all rings considered are commutative with unity and all modules are unital. Let $R$ be a ring and $M$ be an $R$-module. As usual, we use ${\rm pd}_R(M)$, ${\rm id}_R(M)$, and ${\rm fd}_R(M)$ to denote, respectively, the classical projective dimension, injective dimension, and flat dimension of $M$, and ${\rm wdim}(R)$ and ${\rm gldim}(R)$ to denote, respectively, the weak and global homological dimensions of $R$.

In \cite{fu zh}, Fu, Zhu and Ding introduced the concepts of copure projective modules, $n$-copure projective modules, strongly copure projective modules. An $R$-module $M$ is called $n$-copure projective if $\Ext^1_R(M,N)=0$ for any $R$-module $N$ with ${\rm fd}_R(N)\leqslant n$. $M$ is said to be strongly copure projective if $\Ext^i_R(M,N)=0$ for any flat $R$-module $N$, and all $i\geq 1$. The $0$-copure projective modules are called copure projective. Also, the authors in this paper proved that $R$ is a QF ring if and only if every $R$-module is copure projective. For more details about copure projective (strongly copure projective) modules, see  \cite{fu zh, Gao, TAO}.

The motivation of this paper is to generalize the concepts of copure projective in the module case and the ring case using a $w$-operation.

Now, we review some definitions and notation. Let $J$ be an ideal of $R$. Following \cite{HFX}, $J$ is called a \emph{Glaz-Vasconcelos ideal} (a $GV$-ideal for short) if $J$ is finitely generated and the natural homomorphism $\varphi : R \rightarrow J^{\ast} = \Hom_R(J,R)$ is
an isomorphism.  Let $M$ be an $R$-module and define
$${\rm tor}_{GV}(M) = \{x\in M \mid Jx = 0\text{  for some } J\in  GV(R)\},$$
where  $GV(R)$ is the set of $GV$-ideals of $R$. It is clear that ${\rm tor}_{GV}(M)$ is a submodule of $M$. Now $M$ is said to be $GV$-\emph{torsion} (resp., $GV$-\emph{torsion-free}) if ${\rm tor}_{GV}(M) =M$ (resp., ${\rm tor}_{GV}(M) =0$).   A $GV$-torsion-free module $M$ is called a $w$-\emph{module} if ${\rm Ext}^1_R(R/J, M) =0$ for any $J\in GV(R)$. Projective modules and reflexive modules are $w$-modules. In the recent paper \cite{WANG FLAT}, it was shown that flat modules are $w$-modules. The notion of $w$-modules was introduced firstly over a domain \cite{MCCASLAND} in the study of Strong Mori
domains and was extended to commutative rings with zero divisors in \cite{HFX}. Let $w$-${\rm Max}(R)$ denote the set of maximal $w$-ideals of $R$, i.e., $w$-ideals of $R$ maximal among proper integral $w$-ideals of $R$. Following  \cite[Proposition 3.8]{HFX}, every maximal $w$-ideal is prime. For any $GV$-torsion-free module $M$,
$$M_{w}:=\{x\in E(M)\mid Jx\subseteq M \text{  for some } J\in  GV(R)\}$$
is a $w$-submodule of $E(M)$ containing $M$ and is called the $w$-\emph{envelope} of $M$,
where $E(M)$ denotes the injective hull  of $M$. It is clear that a $GV$-torsion-free module $M$ is a $w$-module if and only if $M_{w}=M$.
Let $M$ and $N$ be $R$-modules and let $f : M \rightarrow N$ be a homomorphism. Following \cite{FGTYPE},
$f$ is called a $w$-\emph{monomorphism} (resp., $w$-\emph{epimorphism}, $w$-\emph{isomorphism}) if $f_{\mathfrak{p}} :
M_{\mathfrak{}}\rightarrow N_{\mathfrak{p}}$ is a monomorphism (resp., an epimorphism, an isomorphism) for all
$\mathfrak{p}\in w\text{-}{\rm Max}(R)$.  A sequence $0\rightarrow A \rightarrow B \rightarrow C\rightarrow 0$ of $R$-modules is said to be $w$-exact if $0\rightarrow A_{\p} \rightarrow B_{\p} \rightarrow C_{\p}\rightarrow 0$ is an exact for any $\p\in$ $w$-${\rm Max}(R)$. An $R$-module $M$ is called a $w$-\emph{flat module} if the induced map $1 \otimes f : M\otimes A \rightarrow M\otimes B$ is a $w$-monomorphism for any $w$-monomorphism $f : A\rightarrow B$. Certainly flat modules are $w$-flat, but the converse implication is not true in general.

In Section $2$, the definition and some general results are given. An $R$-module $M$ is said to be $w$-copure projective (resp., strongly $w$-copure projective) if $\Ext_{R}^1(M,N)$ (resp., $\Ext_{R}^n(M,N)$) is a GV-torsion $R$-module for all flat $R$-modules $N$ and any $n\geq1$. In this Section we will prove that if $R$ is a ring with finite weak global dimension ($\wgldim(R)<\infty$), then $R$ is a $DW$-ring if and only if every strongly $w$-copure projective $R$-module is strongly copure projective and we will prove that if $R$ is a Krull domain, then $R$ is QF if and only if every $R$-module is $w$-copure projective (resp., strongly $w$-copure projective) and $R$ is a $DW$-ring if and only if every $w$-copure projective (resp., strongly $w$-copure projective) $R$-module is a strong $w$-module (and so, a $w$-module).

In Section $3$,  we study the $w$-copure projective (resp., strongly $w$-copure projective) modules over coherent rings and we will prove that if $R$ is a domain, then $R$ is a $w$-IF-ring( a ring with every injective $w$-module is $w$-flat) if and only if every finitely presented module is $w$-copure projective (resp., strongly $w$-copure projective). Also, we prove that a ring $R$ is semihereditary if and only if every $w$-copure projective $R$-module is flat if and only if every finitely presented $w$-copure projective $R$-module is projective.

In the last section, we study change of rings theorems for the $w$-copure projective in various contexts.

\section{\bf $w$-copure projective modules}
We start this section with the following definition.
\begin{definition}\label{w-copu proj}
	An $R$-module $M$ is said to be $w$-copure projective if $\Ext_R^1(M,N)$ is  GV-torsion for any flat $R$-module $N$, and $M$ is said to be strongly $w$-copure projective if $\Ext_R^i(M,N)$ is GV-torsion for any flat $R$-module $N$ and any $i\geq 1$.
\end{definition}
Obviously, every copure projective (resp., strongly copure projective) module is $w$-copure projective (resp., strongly $w$-copure projective). Recall that a ring $R$ is called a $DW$-ring if every ideal of $R$ is a $w$-ideal, or equivalently $GV(R)={R}$ \cite[Proposition 2.2]{A Mimo}. In this case, the only GV-torsion module is $(0)$, and so over such a ring, copure projective (resp., strongly copure projective) $R$-module and $w$-copure projective (resp., strongly $w$-copure projective) $R$-module coincide. Examples of $DW$-rings are Pr\"{u}fer domains, domains with Krull dimension one, and rings with Krull dimension zero.
\begin{proposition}\label{w-coup}
	Let $R$ be a ring and $M$ be an $R$-module. If $M$ is strongly $w$-copure projective, then $\Ext_R^1(M,N)$ is GV-torsion for any $R$-module $N$ with ${\rm fd}_R N<\infty$.
\end{proposition}
\begin{proof}
	Let $N$ be an $R$-module with ${\rm fd}_R N=m<\infty$.
	 Then, there is an exact sequence $0\rightarrow F_m\rightarrow F_{m-1}\rightarrow \dots F_0\rightarrow N\rightarrow 0$ with each $F_i$ flat, and so
	$\Ext_R^1(M,N)\cong \Ext_R^{m+1}(M,F_m)$. Since $M$ is strongly $w$-copure projective, $\Ext_R^{m+1}(M,F_m)$ is GV-torsion which implies that $\Ext_R^1(M,N)$ is  GV-torsion.
\end{proof}
	\begin{proposition}\label{loca}
	Let $M$ be a strongly $w$-copure projective module and $\m$ be a maximal $w$-ideal of $R$. If $N$ is an $R_{\m}$-module which is flat as an $R$-module, then $\Ext_R^n(M,N)=0$ for all $n\geq 1$.
\end{proposition}
\begin{proof} Since $N$ as an $R$-module is flat, $\Ext_R^n(M,N)$ is  $GV$-torsion. So $\Ext_R^n(M,N)=\left(\Ext_R^n(M,N)\right)_{\m}=0$.
\end{proof}
Recall from \cite{Wang and Lie}, that an $R$-module $M$ is called $w$-split if $\Ext_{R}^1(M,N)$ is a GV-torsion $R$-module for all $R$-modules $N$. From \cite{REF}, the $w$-split dimension of an $R$-module $M$, $w$-$\rm sd(M)$,  is defined by declaring that $w$-$\rm sd$$(M)\leq n$  ($n\in \mathbb{N}$) if $M$ has a $w$-split  resolution of length $n$. Otherwise, we set $w$-$\rm sd$$(M)=\infty$.

By \cite[Proposition 2.4 and Proposition 2.7]{Wang and Lie}, we have the following lemma.
\begin{lemma}\label{w-split}
	Let $R$ be a ring. The following statements hold:
	\begin{enumerate}
		\item  Every $w$-split $R$-module is strongly $w$-copure
		projective.
		\item  Every $w$-projective $w$-module is strongly $w$-copure	projective.
	\end{enumerate}
\end{lemma}
In \cite{FHT}, the authors introduced the $w$-Nagata ring, $R\{X\}$, of $R$. It is a localization of $R[X]$ at the multiplicative closed set  $S_w:=\{f\in R[X]\;|\;c(f)\in$
$GV(R)\}$, where $c(f)$ denotes the content of $f$. Similarly, the $w$-Nagata module $M\{X\}$ of an $R$-module $M$ is defined as $M\{X\}:= M[X]_{S_w}=M\otimes_R R\{X\}$.

In general, (strongly) $w$-copure projective modules need not be $w$-split as shown by the next proposition, which is a generalization of [\cite{fu zh}, Proposition 3.4].
\begin{proposition}\label{w-split2}
	Let $R$ be a ring and $M$ be an $R$-module. Then the following conditions are equivalent:
	\begin{enumerate}
		\item  $M$ is a $w$-split $R$-module;
		\item  $M$ is a strongly $w$-copure projective $R$-module and $w$-$\rm sd$$(M)< \infty$;
		\item  $M$ is a  $w$-copure projective $R$-module and $w$-$\rm sd$$(M)\leq 1$;
		
		If $M$ is a finitely presented $R$-module, then the above conditions are also equivalent to
		\item  $M\{X\}$ is a projective $R\{X\}$-module.
	\end{enumerate}
\end{proposition}
\begin{proof}
	$(1)\Rightarrow (2)$ and $(1)\Rightarrow (3)$ hold by Lemma \ref{w-split} and \cite[Proposition 3.3]{REF}.
	
	$(2)\Rightarrow (1)$ Let $X$ be an $R$-module and suppose $w$-$\rm sd$$(M)=n< \infty$. Consider an exact sequence $0\rightarrow N\rightarrow F\rightarrow X\rightarrow 0$ with $F$ free, so we have the exact sequence
	$$ \Ext_R^n(M,F)\rightarrow \Ext_R^n(M,X)\rightarrow \Ext_R^{n+1}(M,N).$$
	The left term is GV-torsion since $M$ is a strongly $w$-copure projective $R$-module and the right term is GV-torsion since $w$-$\rm sd$$(M)$ = $n$ by \cite[Proposition 3.3]{REF}. Hence, $\Ext_R^n(M,X)$ is GV-torsion, so $M$ is $w$-split by \cite[Proposition 2.4]{Wang and Lie}.
	
	$(3)\Rightarrow (1)$ Similar to the proof of $(2)\Rightarrow (1)$.
	
	$(1)\Rightarrow (4)$ Trivial by \cite[Theorem 6.7.13]{KIMBOOK}, and since every $w$-split module is $w$-projective.
		
	$(4)\Rightarrow (1)$ Since $M$ is a finitely presented $R$-module, by \cite[Theorem 6.6.24]{KIMBOOK}, $M\{X\}$ is finitely generated. Thus, by \cite[Theorem 6.7.18]{KIMBOOK}, $M$ is $w$-projective and by \cite[Proposition 2.5]{REF} $M$ is $w$-split.
\end{proof}

Next, we will give an example of a $w$-copure projective (resp., strongly $w$-copure projective) module which is not a $w$-projective (and so not $w$-split) module.
\begin{example}\label{example}
	Let $R$ be a QF-ring, but not semisimple. For example, $R:=k[X]/(X^2)$ where $k$ is a field. Then, every flat module is injective (since, $R$ is QF) so every $R$-module is $w$-copure projective (resp., strongly $w$-copure projective). Since $R$ is not semisimple, there exists an $R$-module which is not $w$-projective by \cite[Theorem 3.15]{FHT}, (and so not $w$-split).	
\end{example}
Recall from \cite{wange qi} that an $R$-module $M$ is called a strong $w$-module if $\Ext^i_R(N,M)=0$ for any GV-torsion module $N$ and any $i \geq 1$. In \cite{wange qi}, $\PP_{w}^{\dagger}$ denotes the class of GV-torsion-free $R$-modules $N$ with the property that $\Ext^k_R(M,N)=0$ for all $w$-projective $R$-modules $M$ and for all integers $k\geq1$ Clearly, every GV-torsionfree injective $R$-module belongs to $\PP_{w}^{\dagger}$. We note that if $R$ is a $DW$ ring, then every  $R$-module is in $\PP_{w}^{\dagger}$.
\begin{proposition}\label{carw4}
	Let $R$ be a ring. The following statements are equivalent:
	\begin{enumerate}
		\item  $R$ is a $DW$ ring;
		\item  Every (strongl) $w$-copure projective $R$-module is  a strong $w$-module;
		\item  Every (strongl) $w$-copure projective $R$-module is a $w$-module.
	\end{enumerate}
\end{proposition}
\begin{proof}
	$(1)\Rightarrow (2)$ Since $R$ is $DW$, every $R$-module is in $\PP_{w}^{\dagger}$. By \cite[Proposition 2.3]{wange qi}, every $R$-module is a strong $w$-module.
	
	$(2)\Rightarrow (3)$ Trivial, since every strong $w$-module is a $w$-module.

	$(3)\Rightarrow (1)$ Let $J \in GV(R)$, so $R/J$ is $w$-split by \cite[Corollary 2.7]{REF} . Hence, $R/J$ is (strongl) $w$-copure projective by Lemma \ref{w-split}. Hence, $R/J$ is a $w$-module (and so GV-torsion-free). Then, $R/J=0$, so $GV(R)=\{R\}$. Hence, $R$ is a $DW$ ring.
\end{proof}

The next result characterizes rings with finite weak global dimensions over which every strongly $w$-copure projective module is strongly copure projective.
\begin{theorem}\label{carw5}
	Let $R$ be a ring with finite weak global dimension. The following statements are equivalent:
	\begin{enumerate}
		\item  $R$ is a $DW$ ring;
		\item  Every strongly $w$-copure projective $R$-module is projective;
		\item  Every strongly $w$-copure projective $R$-module is strongly copure projective.
	\end{enumerate}
\end{theorem}
\begin{proof}
	$(1)\Rightarrow (2)$ Let $M$ be a strongly $w$-copure projective $R$-module, so $M$ is strongly copure projective since $R$ is $DW$. Then, by \cite[Proposition 3.4]{fu zh}, we have $M$ is a projective $R$-module.
	
	$(2)\Rightarrow (3)$ Trivial, since every projective $R$-module is strongly copure projective.
	
	$(3)\Rightarrow (1)$ Let $J \in GV(R)$, so $R/J$ is $w$-split by \cite[Corollary 2.7]{REF}. Hence, $R/J$ is strongly $w$-copure projective by Lemma \ref{w-split} and so  strongly copure projective. Hence, by \cite[Proposition 3.4]{fu zh}, $R/J$ is projective (and so GV-torsion-free). Then, $R/J=0$, so $GV(R)=\{R\}$. Thus, $R$ is a $DW$ ring.
\end{proof}
\begin{remark}\label{remark}
	\begin{enumerate}
		\item The condition $\wgldim(R)<\infty$ in Theorem \ref{carw5} is necessary, because there exists a $DW$ ring $R$ such that a strongly $w$-copure projective $R$-module is not projective (for example QF ring is not semisimple, see Example \ref{example}).
		
		\item  In general, we can not obtain every strongly $w$-copure projective $R$-module is projective under the condition that $\wgldim(R)<\infty$. For example, let $(R,\m)$ be a regular local ring with $\gldim(R)=n \geq 2$, so $\wgldim(R)<\infty$. Suppose that every strongly $w$-copure projective $R$-module is projective so strongly $w$-copure projective $R$-module is a $w$-module. Hence, by Proposition \ref{carw4}, $R$ is a $DW$ ring, and this is a contradiction by \cite[Example 2.6]{TRB}.
	\end{enumerate}
\end{remark}
In the next example, we will give an example of a strongly $w$-copure projective module, which is not strongly copure projective.

\begin{example}\label{example1} Let $(R, \m)$ be a regular local ring with ${\rm gldim}(R) =n$ ($n\geq 2$). By \cite[Example 2.6]{TRB}, $R$ is not a $DW$ ring. Hence, by Theorem \ref{carw5}, there exists a strongly $w$-copure projective module which is not a strongly copure projective module. Otherwise, if every strongly $w$-copure projective module is strongly copure projective, $R$ is $DW$ by Theorem \ref{carw5}, a contradiction.
\end{example}

\begin{proposition}\label{exactpure}
	Let $R$ be a  ring and  $\xymatrix{ 0\rightarrow A\rightarrow B\rightarrow C\rightarrow 0}$ be an exact sequence of $R$-modules, where $C$ is a strongly ($w$-copure projective) $R$-module. Then, $A$ is a strongly ($w$-copure projective) $R$-module if and only if $B$ is a strongly ($w$-copure projective) $R$-module.
\end{proposition}
\begin{proof}
	Let $N$ be a flat $R$-module. Then, we have the exact sequence,
	$$\xymatrix{ \Ext_{R}^{i}(C,N)\rightarrow \Ext_{R}^{i}(B,N)\rightarrow \Ext_{R}^{i}(A,N)\rightarrow \Ext_{R}^{i+1}(C,N).}$$
	By [\cite{FHT}, Lemma 0.1], for any maximal $w$-ideal $\m$, we have the exact sequence,
	$$\xymatrix{ 0=(\Ext_{R}^{i}(C,N))_{\m}\rightarrow (\Ext_{R}^{i}(B,N))_{\m}\rightarrow (\Ext_{R}^{i}(A,N))_{\m}\rightarrow (\Ext_{R}^{i+1}(C,N))_{\m}}=0.$$
	Thus, $(\Ext_{R}^{i}(A,N))_{\m}\cong (\Ext_{R}^{i}(B,N))_{\m}$. Hence, $\Ext_{R}^{i}(A,N)$ is a $GV$-torsion $R$-module if and only if $\Ext_{R}^{i}(B,N)$ is a $GV$-torsion $R$-module (by [\cite{FHT}, Lemma 0.1]). Thus, $A$ is a strongly $w$-copure projective $R$-module if and only if $B$ is a strongly $w$-copure projective $R$-module.
\end{proof}

\begin{proposition}\label{class}
	Let $R$ be a ring. Then the class of all $w$-copure projective (resp., strongly $w$-copure projective) $R$-modules is closed under direct sums and direct summands.
\end{proposition}
\begin{proof}
	Let $M$ and $N$ be two $R$-modules, and $F$ a flat $R$-module. By \cite[Theorem 3.3.9]{KIMBOOK}, $\Ext_{R}^{i}(M\oplus N,F)\cong\Ext_{R}^{i}(M,F)\oplus\Ext_{R}^{i}(N,F)$ for any $i\geq1$. Hence, by \cite[Lemma 0.1]{FHT}, $\Ext_{R}^{i}(M\oplus N,F)$ is GV-torsion if and only if $\Ext_{R}^{i}(M,F)$ and $\Ext_{R}^{i}(N,F)$ are GV-torsion. Thus, $M\oplus N$ is $w$-copure projective (resp., strongly $w$-copure projective) if and only if $M$ and $N$ are $w$-copure projective (resp., strongly $w$-copure projective).
\end{proof}

\begin{proposition}\label{caracst}
	Let $R$ be a ring. The following statements are equivalent:
	\begin{enumerate}
		\item  $M$ is $w$-copure projective (resp., strongly $w$-copure projective);
		\item  $\Hom_{R}(P,M)$ is $w$-copure projective (resp., strongly $w$-copure projective) for any finitely generated projective $R$-module $P$;
		\item  $P\otimes M$ is $w$-copure projective (resp., strongly $w$-copure projective) for any finitely presented projective $R$-module $P$.
	\end{enumerate}
\end{proposition}
\begin{proof}
	$(1)\Rightarrow (2)$ Let $P$ be a finitely generated projective $R$-module and let $F$ be a flat $R$-module. By  \cite[Theorem 3.3.12]{KIMBOOK},  $P\otimes \Ext_{R}^n(M,F)\cong \Ext_{R}^n(\Hom_{R}(P,M),F)$ for any $n\geq 1$. Since, $M$ is $w$-copure projective (resp., strongly $w$-copure projective), $\Ext_{R}^n(M,F)$ is GV-torsion. Hence, $\Ext_{R}^n(\Hom_{R}(P,M),F)$ is GV-torsion for any $n\geq 1$. Thus, $\Hom_{R}(P,M)$ is $w$-copure projective (resp., strongly $w$-copure projective).
	
	$(1)\Rightarrow (3)$ Let $P$ be a finitely presented projective $R$-module and let $F$ be a flat $R$-module. By  \cite[Theorem 3.3.10]{KIMBOOK}, $\Ext_{R}^n(P\otimes M,F)\cong \Hom_{R}(P,\Ext_{R}^n(M,F))$ for any $n\geq 1$. Hence, $(\Ext_{R}^n(P\otimes M,F))_{\m}\cong (\Hom_{R}(P,\Ext_{R}^n(M,F)))_{\m}$
	for any maximal $w$-ideal $\m$ of $R$ and by \cite[Theorem 2.6.16]{KIMBOOK}, $$(\Ext_{R}^n(P\otimes M,F))_{\m}\cong \Hom_{R_{\m}}(P_{\m},(\Ext_{R}^n(M,F))_{\m}).$$  By \cite[Theorem 6.2.15]{KIMBOOK} and since $M$ is $w$-copure projective (resp., strongly $w$-copure projective), we have $(\Ext_{R}^n(M,F))_{\m}=0$. Hence, $\Hom_{R_{\m}}(P_{\m},(\Ext_{R}^n(M,F))_{\m})=0$, and so $(\Ext_{R}^n(P\otimes M,F))_{\m}=0$. By \cite[Theorem 6.2.15]{KIMBOOK} again we have $\Ext_{R}^n(P\otimes M,F)$ is GV-torsion, which implies that, $ P\otimes M$ is $w$-copure projective (resp., strongly $w$-copure projective).
	
	$(2)\Rightarrow (1)$ and $(3)\Rightarrow (1)$ These follow immediately by taking $P:=R$.
\end{proof}

\begin{lemma}\label{lemma}
	Let $R$ be a ring and  $M$ be an $R$-module. The following statements are equivalent:
	\begin{enumerate}
		\item  $M$ is $w$-copure projective;
		\item  For any exact sequence $0\rightarrow A\rightarrow B\rightarrow C\rightarrow 0$  with $A$ flat, the exact sequence $0\rightarrow \Hom_{R}(M,A)\rightarrow \Hom_{R}(M,B)\rightarrow \Hom_{R}(M,C)\rightarrow 0$ is $w$-exact;
		\item  For any exact sequence $0\rightarrow L\rightarrow E\rightarrow M\rightarrow 0$  and any flat $R$-module $N$, the exact sequence $0\rightarrow \Hom_{R}(M,N)\rightarrow \Hom_{R}(E,N)\rightarrow \Hom_{R}(L,N)\rightarrow 0$ is $w$-exact.
	\end{enumerate}
\end{lemma}
\begin{proof}
	$(1)\Rightarrow (2)$ Let $\m$ be a maximal $w$-ideal of $R$. Tensoring the exact sequence $0\rightarrow\Hom_{R}(M,A)\rightarrow \Hom_{R}(M,B)\rightarrow \Hom_{R}(M,C) \rightarrow \Ext_{R}^1(M,A)$ with $R_{\m}$, we have the exact sequence $0\rightarrow(\Hom_{R}(M,A))_{\m}\rightarrow (\Hom_{R}(M,B))_{\m}\rightarrow (\Hom_{R}(M,C))_{\m}\rightarrow (\Ext_{R}^1(M,A))_{\m}$. Since $M$ is $w$-copure projective, $\Ext_{R}^1(M,A)$ is GV-torsion and by \cite[Theorem 6.2.15]{KIMBOOK}, $(\Ext_{R}^1(M,A))_{\m}=0$. Hence, the sequence $0\rightarrow(\Hom_{R}(M,A))_{\m}\rightarrow (\Hom_{R}(M,B))_{\m}\rightarrow (\Hom_{R}(M,C))_{\m}\rightarrow 0$ is exact and so, $0\rightarrow\Hom_{R}(M,A)\rightarrow \Hom_{R}(M,B)\rightarrow \Hom_{R}(M,C)\rightarrow 0$ is $w$-exact.
	
	$(2)\Rightarrow (1).$ Let $N$ be a flat $R$-module, and consider an exact sequence $0\rightarrow N\rightarrow E\rightarrow A\rightarrow 0$ where $E$ is an injective module. Thus, by tensoring the exact sequence
	$0\rightarrow {\rm Hom}_R(M,N) \rightarrow {\rm Hom}_R(M,E)\rightarrow {\rm Hom}_R(M,A)\rightarrow {\rm Ext}_R^1(M,N)\rightarrow0$, with $R_{\mathfrak{m}}$ (for any maximal $w$-ideal $\mathfrak{m}$ of $R$), and keeping in mind that $0\rightarrow {\rm Hom}_R(M, N)\rightarrow {\rm Hom}_R(M,E)\rightarrow {\rm Hom}_R(M,A)\rightarrow 0$ is $w$-exact, we can deduce that ${\rm Ext}_R^1(M,N)$ is $GV$-torsion, which implies that $M$ is $w$-copure projective.
	
	$(1)\Rightarrow (3)$ Let $0\rightarrow L\rightarrow E\rightarrow M\rightarrow 0$ be an exact sequence. For any flat $R$-module $N$ and any maximal $w$-ideal $\m$ of $R$, we have the exact sequence $0\rightarrow(\Hom_{R}(M,N))_{\m}\rightarrow (\Hom_{R}(E,N))_{\m}\rightarrow (\Hom_{R}(L,N))_{\m}\rightarrow (\Ext_{R}^1(M,N))_{\m}$. Since $M$ is  $w$-copure projective, $\Ext_{R}^1(M,N)$ is GV-torsion and by \cite[Theorem 6.2.15]{KIMBOOK}, $(\Ext_{R}^1(M,N))_{\m}=0$. Hence, we have  the exact sequence $0\rightarrow(\Hom_{R}(M,N))_{\m}\rightarrow (\Hom_{R}(E,N))_{\m}\rightarrow (\Hom_{R}(L,N))_{\m}\rightarrow 0$ and so the sequence $0\rightarrow \Hom_{R}(M,N)\rightarrow \Hom_{R}(E,N)\rightarrow \Hom_{R}(L,N)\rightarrow 0$ is $w$-exact.
	
	$(3)\Rightarrow (1)$ Let  $0\rightarrow K\rightarrow P\rightarrow M\rightarrow 0$ be an exact sequence with $P$ projective. For any flat $R$-module $N$ we have the exact sequence, $0\rightarrow \Hom_{R}(M,N)\rightarrow \Hom_{R}(P,N)\rightarrow \Hom_{R}(K,N)\rightarrow \Ext_{R}^1(M,N)\rightarrow 0$
	and keeping in mind that $0\rightarrow \Hom_{R}(M,N)\rightarrow \Hom_{R}(P,N)\rightarrow \Hom_{R}(K,N)\rightarrow 0$ is $w$-exact, we can deduce by \cite[Theorem 6.2.15]{KIMBOOK} that $\Ext_{R}^1(M,N)$ is GV-torsion. Thus, $M$ is $w$-copure projective.
\end{proof}

In \cite{wange qi}, Wang and Qiao defined the strong $w$-module as the following: A GV-torsion-free module $M$ is said to be a strong $w$-module if $\Ext^i_R(N,M)=0$ for each integer $i\geq 1$ and for all GV-torsion modules $N$, so every strong $w$-module is $w$-module. Hence, by this definition we have the following result.
\begin{proposition}\label{prop}
	Let $0\rightarrow A\rightarrow F\rightarrow C\rightarrow 0$ be a $w$-exact sequence of $R$-modules with $F$ strongly $w$-copure projective. Then, for any strong $w$-module flat $N$ and all integers $n\geq 1$, $\Ext_{R}^n(A,N)$ is GV-torsion if and only if so is $ \Ext_{R}^n(C,N)$.
\end{proposition}
\begin{proof}
	Let $0\rightarrow A\rightarrow F\rightarrow C\rightarrow 0$ be a $w$-exact sequence, by \cite[Lemma 2.1]{wange qi}, for any strong $w$-module flat $N$, there exists an exact sequence $$(*)\cdots\cdots\cdots\cdots\cdots \Ext_{R}^n(F,N)\rightarrow \Ext_{R}^n(A,N) \rightarrow\Ext_{R}^{n+1}(C,N)\rightarrow \Ext_{R}^{n+1}(F,N)$$
	Let $\m$ be a maximal $w$-ideal of $R$. Tensoring the exact sequence $(*)$ with $R_{\\m}$, we have the exact sequence
	$(\Ext_{R}^n(F,N))_{\m}\rightarrow (\Ext_{R}^n(A,N))_{\m} \rightarrow (\Ext_{R}^{n+1}(C,N))_{\m}\rightarrow (\Ext_{R}^{n+1}(F,N))_{\m}$. Since $F$ is strongly $w$-copure projective, the left and the right term equal zero. Hence, $(\Ext_{R}^n(A,N))_{\m}\cong (\Ext_{R}^{n+1}(C,N))_{\m}$. Thus, by \cite[Theorem 6.2.15]{KIMBOOK} $\Ext_{R}^n(A,N)$ is GV-torsion if and only if $\Ext_{R}^{n+1}(C,N)$ is GV-torsion.
\end{proof}

Recall from \cite[Theorem 2.6]{SXFW} that an $R$-module $A$ is said to be absolutely $w$-pure if and only if $\Ext_R^1(N,A)$ is a $GV$-torsion $R$-module for every finitely presented $R$-module $N$.

\begin{proposition}\label{carwqf2}
	Let $R$ be a coherent ring. The following statements are equivalent:
	\begin{enumerate}
		\item  Every $R$-module is strongly $w$-copure projective;
		\item  Every finitely presented $R$-module is strongly $w$-copure projective;
		\item  Every flat $R$-module is absolutely $w$-pure.		
	\end{enumerate}
\end{proposition}
\begin{proof}
	$(1)\Rightarrow (2)$ Obvious.
	
	$(2)\Rightarrow (3)$ Let $N$ be a flat $R$-module. For any finitely presented $R$-module $F$, we have $\Ext_{R}^i(F,N)$ is GV-torsion for any $i\geq1$. Hence, $N$ is absolutely $w$-pure by [\cite{RBT}, Lemma 2].
	
	$(3)\Rightarrow (1)$ Let $M$ be a finitely presented $R$-module. For any flat $R$-module $N$, we have by [\cite{RBT}, Lemma 2] again, $\Ext_{R}^i(M,N)$ is GV-torsion for any $i\geq1$. Thus, $N$ is $w$-copure projective.
\end{proof}

We end this section with the following characterizations of a QF ring (a ring $R$ is called quasi-Frobenius ring if $R$ is a self-injective Noetherian ring).
\begin{proposition}\label{carwqf}
	Let $R$ be a ring. The following statements are equivalent:
	\begin{enumerate}
		\item  Every $R$-module is strongly $w$-copure projective;
		\item  Every $R$-module is $w$-copure projective;
		\item  Every finitely presented $R$-module is $w$-copure projective;
		\item  Every flat $R$-module is $w$-injective;
		\item  Every flat $R$-module is absolutely $w$-pure.		
		
		If $R$ is a Krull domain, then the above four conditions are equivalent to:
		
		\item  $R$ is QF.
	\end{enumerate}
\end{proposition}
\begin{proof}
	$(1)\Rightarrow (2)$ Obvious.
	
	$(2)\Rightarrow (4)$ Let $F$ be a flat $R$-module. Then $F$ is a $w$-module by \cite[Theorem 6.7.24]{KIMBOOK}. For any $R$-module $M$, we have $\Ext_{R}^1(M,F)$ is GV-torsion since $M$ is $w$-copure projective. Hence, by \cite[Theorem 3.3]{Wang and Kim2}, $F$ is a $w$-injective module.
	
	$(4)\Rightarrow (1)$ Let $M$ be an $R$-module. For any flat $R$-module $F$, $F$ is a $w$-injective module by $(3)$. Hence, $\Ext_{R}^n(M,F)$ is GV-torsion by \cite[Theorem 3.3]{Wang and Kim2}. Thus, $M$ is strongly $w$-copure projective.
	
	$(3)\Rightarrow (5)$ Let $N$ be a flat $R$-module. For any finitely presented $R$-module $F$, we have $\Ext_{R}^1(F,N)$ is GV-torsion. Hence, $N$ is absolutely $w$-pure.
	
	$(5)\Rightarrow (3)$ Let $M$ be a finitely presented $R$-module. For any flat $R$-module $N$, we have $\Ext_{R}^1(M,N)$ is GV-torsion. Thus, $N$ is $w$-copure projective.
	
	$(2)\Rightarrow (5)$ Let $M$ be a flat $R$-module. For any finitely presented $R$-module $N$, we have $N$ is $w$-copure projective by $(2)$. Hence, $\Ext_{R}^1(M,N)$ is GV-torsion. Thus, $M$ is an absolutely $w$-pure module by \cite[Theorem 2.6]{SXFW}.
	
	$(5)\Rightarrow (4)$ Let $M$ be a flat module, so $M$ is absolutely $w$-pure by $(5)$. Hence, $M$ is divisible by \cite[Corollary 2.8]{SXFW} and by \cite[Theorem 2.6]{Said  Kim  Wang} we have $M$ is injective. Thus, $M$ is a $w$-injective module.

	$(4)\Rightarrow (6)$ Let $M$ be a projective $R$-module. Then $M$ is a $w$-injective module by $(4)$, and so $M$ is a $w$-module by \cite[Theorem 6.7.24]{KIMBOOK}. Thus, by \cite[Theorem 3.9]{Wang and Kim2}, $M$ is divisible and by \cite[Theorem 2.6]{Said  Kim  Wang} we have $M$ is injective. Hence, $R$ is QF by \cite[Theorem 4.6.10]{KIMBOOK}.
	
	$(6)\Rightarrow (4)$ Let $M$ be a flat module. Then $M$ is injective (since $R$ is QF), and so $M$ is $w$-injective.
\end{proof}

\section{\bf The $w$-copure projective module over coherent rings}
In this section, we study the $w$-copure projective (resp., strongly $w$-copure projective) modules over a coherent rings.

The next proposition is a $w$-theoretic analog of \cite[Proposition 3.9]{fu zh}.
\begin{proposition}\label{propo}
	Let $R$ be a coherent ring and $M$ be a finitely presented $R$-module. Then, $M$ is strongly $w$-copure projective if and only if $\Ext_{R}^n(M,R)$ is GV-torsion for any $n\geq 1$.
\end{proposition}
\begin{proof}
	If $M$ is a strongly $w$-copure projective $R$-module, then $\Ext_{R}^n(M,R)$ is GV-torsion for any $n\geq 1$ by definition. Conversely, assume that $\Ext_{R}^n(M,R)$ is GV-torsion for any $n\geq 1$. Since  every flat module is a direct limit of finitely generated free modules by \cite[Theorem 2.6.20]{KIMBOOK}. Let $F$ be a flat $R$-module, write $F=\varinjlim P_i$ with $P_i$ finitely generated free, so by \cite[Theorem 3.9.4]{KIMBOOK} we have $\Ext_{R}^n(M,F)=\Ext_{R}^n(M,\varinjlim P_i) \cong \varinjlim\Ext_{R}^n(M,P_i)$. Let $\m$ be a maximal $w$-ideal of $R$, by \cite[Lemma 8]{zho wang}, $(\Ext_{R}^n(M,F))_{\m}=(\varinjlim\Ext_{R}^n(M,P_i))_{\m}=\varinjlim(\Ext_{R}^n(M,P_i))_{\m}=0$. Hence, $(\Ext_{R}^n(M,F))_{\m}=0$ and by \cite[Theorem 6.2.15]{KIMBOOK}, $\Ext_{R}^n(M,F)$ is GV-torsion. Hence, $M$ is strongly $w$-copure projective.
\end{proof}
By the proof of Proposition \ref{propo}, we have the following corollary.

\begin{corollary}\label{coro}
	Let $R$ be a coherent ring and $M$ be a finitely presented $R$-module. Then, $M$ is $w$-copure projective if and only if $\Ext_{R}^1(M,R)$ is GV-torsion.
\end{corollary}

In \cite{BKM}, the authors defined the $w$-copure flat (resp., strongly $w$-copure flat) modules as follows: an $R$-module $M$ is called $w$-copure flat if $\Tor^R_1(E,M)$ is a GV-torsion $R$-module for any injective $w$-module $E$, and $M$ is said to be strongly $w$-copure flat if $\Tor^R_n(E,M)$ is a GV-torsion $R$-module for any injective $w$-module $E$ and any $n\geq 1$. The authors also defined the $w$-IF-ring as follows: A ring $R$ is called a $w$-IF-ring if every injective $w$-module is $w$-flat.

Next we will give some characterizations of $w$-IF-rings by (strongly) $w$-copure projective modules.

\begin{proposition}\label{carwIF}
	Let $R$ be a coherent domain. The following conditions are equivalent:
	\begin{enumerate}
		\item  $R$ is a $w$-IF-ring;
		\item  Every finitely presented $R$-module is copure projective (resp., strongly copure projective);
		\item  Every finitely presented $R$-module is $w$-copure projective (resp., strongly $w$-copure projective);
		\item  $R$ is absolutely $w$-pure.
	\end{enumerate}
\end{proposition}
\begin{proof}
	$(1)\Rightarrow (2)$ Let $M$ be a finitely presented $R$-module. By \cite[Corollary 3.6 and Proposition 3.10]{BKM}, $M$ is $w$-copure flat (resp., strongly $w$-copure flat) and $R$ is a field. Thus, $R$ is $DW$ by \cite[Example 1.8.5 and Corollary 6.3.13]{KIMBOOK}. Hence, $M$ is copure flat (resp., strongly copure flat) and by \cite[Proposition 3.8]{fu zh}, we have $M$ is copure projective (resp., strongly copure projective).
	
	$(2)\Rightarrow (3)$ Trivial.
	
	$(3)\Rightarrow (1)$ Let $N$ be an injective $w$-module. For any finitely presented $R$-module $F$, we have $F$ is $w$-copure projective (resp., strongly $w$-copure projective) by $(3)$. Hence, by Proposition \ref{propo} and Corollary \ref{coro}, we have $\Ext_{R}^n(F,R)$ is GV-torsion for any $n\geq1$. Thus, by \cite[Theorem 3.9.3]{KIMBOOK}, we have $$\Tor^{R}_n(F,\Hom_{R}(R,N))\cong \Hom_{R}(\Ext_{R}^n(F,R),N).$$
	Thus, by \cite[Exersice 6.22]{KIMBOOK}, $\Hom_{R}(\Ext_{R}^n(F,R),N)=0$.
	Hence, $\Tor^{R}_n(F,N)\cong \Hom_{R}(\Ext_{R}^n(F,R),N)=0$. Thus, $N$ is a flat $R$-module by \cite[Theorem 3.4.6]{KIMBOOK}, which implies that $N$ is a $w$-flat $R$-module. Hence, $R$ is a $w$-IF-ring.
	
	$(3)\Rightarrow (4)$ By Proposition \ref{carwqf2}, we have $R$ is absolutely $w$-pure.
	
	$(4)\Rightarrow (3)$ Let $M$ be a finitely presented $R$-module. Since $R$ is absolutely $w$-pure, $\Ext_{R}^1(F,R)$ is GV-torsion. Hence, by Proposition \ref{propo}, we have $M$ is strongly $w$-copure projective.
\end{proof}

\begin{proposition}\label{coro4}
	Let $R$ be a Krull domain. Then, $R$ is QF if and only if $R$ is a $w$-IF-ring.
\end{proposition}
\begin{proof}
	Let $R$ be a QF ring. Hence $R$ is an IF-ring and so a $w$-IF-ring. Conversely, if $R$ is a $w$-IF-ring, then by \cite[Proposition 3.10]{BKM} $R$ is a field and so coherent. Hence, by Proposition \ref{carwIF}, Proposition \ref{carwqf2} and Proposition \ref{carwqf}, we have $R$ is QF.
\end{proof}

\begin{proposition}\label{propo4}
	Let $R$ be a coherent ring and $M$ be a finitely presented $R$-module. Then,
	$M$ is a strongly $w$-copure projective $R$-module if and only if $M_{\m}$ is a strongly copure projective $R_{\m}$-module for any maximal $w$-ideal $\mathfrak{m}$ of $R$.

\end{proposition}
\begin{proof}
	Let $N$ be a flat $R_{\m}$-module, then $N$ is a flat $R$-module (see, [\cite{KIMBOOK}, Theorem 3.8.5]). Hence, $\Ext^n_R(M,N)=0$ by Proposition \ref{loca}. Since $\Ext^n_R(M,N)$ is an $R_{\m}$-module,  by [\cite{KIMBOOK}, Theorem 3.9.11], we have
	$$ \Ext_R^n(M_{\m},N)\cong \Ext^n_R(M,N)_{\m}\cong\Ext_R^n(M,N)=0,$$ which implies that $M_{\m}$ is a strongly copure projective $R_{\m}$-module.
		
	Conversely, Let $M_{\m}$ be a strongly copure projective $R$-module for any maximal $w$-ideal $\m$ of $R$ and $N$ be a flat $R$-module, so by \cite[Theorem 2.5.12]{KIMBOOK}, $N_{\m}$ is a flat $R_{\m}$-module. Hence, by \cite[Theorem 3.9.11]{KIMBOOK} $\Ext_R^1(M,N)_{\m}\cong \Ext_{R_{\m}}^1(M_{\m},N_{\m})=0$  since $M_{\m}$ is a strongly copure projective $R_{\m}$-module. Thus,  $(\Ext_{R}^n(M,N))_{\m}=0$ which implies that $\Ext_{R}^n(M,N)$ is GV-torsion by \cite[Theorem 6.2.15]{KIMBOOK}. Hence, $M$ is strongly $w$-copure projective.
\end{proof}	
The next proposition is a $w$-theoretic analog of \cite[Proposition 3.7]{fu zh}.
\begin{proposition}\label{carw}
	Let $R$ be a coherent ring. Then,  the following conditions hold:
	\begin{enumerate}
		\item  Every (resp., strongly) $w$-copure projective $R$-module is (resp., strongly) $w$-copure flat.
		\item  Every finitely presented (resp., strongly) $w$-copure flat $R$-module is (resp., strongly) $w$-copure projective.
	\end{enumerate}
\end{proposition}
\begin{proof}
	$(1)$ Let $N$ be an injective $w$-module and $E$ be an injective $R_{\m}$-module for any maximal $w$-ideal $\m$ of $R$. So, $E$ is an injective $R$-module by \cite[Exersice 3.16]{KIMBOOK}. Thus, $\Hom_{R}(N,E)$ is a flat $R$-module by \cite[Theorem 4.6.14]{KIMBOOK}. Let $M$ be a (resp., strongly) $w$-copure projective $R$-module, by \cite[Theorem 3.4.11]{KIMBOOK}, we have
	$$\Hom_{R}(\Tor^R_n(M,N),E)\cong \Ext_{R}^n((M,\Hom_{R}(N,E)).$$

Since $M$ is (resp., strongly) $w$-copure projective, we have $\Ext_{R}^n((M,\Hom_{R}(N,E))$ is GV-torsion. Therefore $\Hom_{R}(\Tor^R_n(M,N),E)$ is GV-torsion. It follows that $(\Hom_{R}(\Tor^R_n(M,N),E))_{\m}=0$ by \cite[Theorem 6.2.15]{KIMBOOK}. Consequently, we have $\Hom_{R_{\m}}((\Tor^R_n(M,N))_{\m},E)=0$ by \cite[Theorem 6.7.10]{KIMBOOK} and by \cite[Proposition 2.4.21]{KIMBOOK}, $(\Tor^R_n(M,N))_{\m}=0$ and by \cite[Theorem 6.2.15]{KIMBOOK} again $\Tor^R_n(M,N)$ is GV-torsion for any $n\geq 1$. Hence, $M$ is (resp., strongly) $w$-copure flat.
	
	$(2)$ Let $N$ be a flat $R$-module and $E$ be an injective $R_{\m}$-module for any maximal $w$-ideal $\m$ of $R$. So, $E$ is an injective $R$-module by \cite[Exersice 3.16]{KIMBOOK} and by \cite[Proposition 6.2.18]{KIMBOOK}, we have $E$ is a $w$-module. Hence, $\Hom_{R}(N,E)$ is an injective $w$-module by \cite[Theorem 2.5.5 and Theorem 6.1.18]{KIMBOOK}. Let $M$ be a finitely presented (resp., strongly) $w$-copure flat $R$-module. By \cite[Theorem 3.9.3]{KIMBOOK}, $$\Hom_{R}(\Ext_{R}^n(M,N),E)\cong \Tor^R_n(M,\Hom_{R}(N,E)).$$
Since $M$ is (resp., strongly) $w$-copure flat, $\Tor^R_n(M,\Hom_{R}(N,E))$ is GV-torsion and so $\Hom_{R}(\Ext_{R}^n(M,N),E)$ is GV-torsion. Hence, $(\Hom_{R}(\Ext_{R}^n(M,N),E))_{\m}=0$ by \cite[Theorem 6.2.15]{KIMBOOK} and  $(\Hom_{R}((\Ext_{R}^n(M,N))_{\m},E))=0$ by \cite[Theorem 6.7.10]{KIMBOOK}. Thus, by \cite[Proposition 2.4.21]{KIMBOOK}, $(\Ext_{R}^n(M,N))_{\m}=0$. Hence, by \cite[Theorem 6.2.15]{KIMBOOK} again $\Ext_{R}^n(M,N)$ is GV-torsion. Then, $M$ is (resp., strongly) $w$-copure projective.
\end{proof}

In the final result of this Section we will give some characterizations of semihereditary rings by the $w$-copure projective modules.
Recall that a ring $R$ is said to be semihereditary if every finitely generated ideal of $R$ is projective.
\begin{proposition}\label{semiher}
	Let $R$ be a coherent ring. The following conditions are equivalent:
	\begin{enumerate}
		\item  $R$ is  semihereditary;
		\item  Every $w$-copure projective $R$-module is flat;
		\item  Every finitely presented $w$-copure projective $R$-module is projective.
	\end{enumerate}
\end{proposition}
\begin{proof}
	$(1)\Rightarrow (2)$ Let $M$ be a $w$-copure projective $R$-module. Then $M$ is copure projective (since $R$ is $DW$). Thus, $M$ is flat by \cite[Proposition 2.5]{TAO}.
	
	$(2)\Rightarrow (3)$ Let $M$ be a finitely presented $w$-copure projective $R$-module. Then $M$ is flat by $(2)$. Hence, $M$ is projective (since $M$ is finitely presented).
	
	$(3)\Rightarrow (1)$ Let $M$ be a finitely presented copure projective $R$-module. Then $M$ is finitely presented $w$-copure projective. Thus, by $(3)$ $M$ is projective. Hence, by \cite[Proposition 2.5]{TAO}, $R$ is  semihereditary.
\end{proof}

\section{\bf Change of rings theorems for the $w$-copure projective}

In this Section, we study change of rings theorems for the $w$-copure projective in various contexts.
\begin{definition}[\cite{BKM}, Definition 4.1]\label{def bouba}
	Let $\phi: R \rightarrow T$ be a ring homomorphism. Then $T$ is said to have property ($B_\phi$) if the following property is satisfied: Let $N$ be a $T$-module. If $N$ is a GV-torsion $R$-module, then $N$ is also a GV-torsion $T$-module.
\end{definition}
\begin{proposition}\label{propo2}
	Let $\phi:R\rightarrow T$ be a ring homomorphism such that $T$ as an $R$-module is a flat module, and $T$ has property ($B_\phi$). If $M$ is a $w$-copure projective $R$-module, then $T\otimes_R M$ is a $w$-copure projective $T$-module.
\end{proposition}
\begin{proof}
	Let $0 \rightarrow A\rightarrow P \rightarrow M \rightarrow 0$ be an exact sequence, where $P$ is a projective $R$-module. Then, $0 \rightarrow T \otimes_R A \rightarrow T\otimes_R P \rightarrow T \otimes_R M \rightarrow 0$ is an exact sequence. For any flat $T$-module $F$, we have the following commutative diagram with exact rows:
	$$\begin{array}{ccccc}
	{\rm Hom}_{T}(T\otimes_R P, F) & \rightarrow & {\rm Hom}_{T}(T \otimes_R A,F ) & \rightarrow & {\rm Ext}_{T}^1(T \otimes_R M,F) \\
	\downarrow &  & \downarrow &  & \downarrow \\
	{\rm Hom}_{R}(P, F) & \rightarrow & {\rm Hom}_{R}(A,F) & \rightarrow & {\rm Ext}_{R}^1(M,F)
	\end{array}$$
By the adjoint isomorphism theorem, two vertical arrows on the left side are isomorphisms. Thus, ${\rm Ext}_{T}^1(T \otimes_R M,F)\cong {\rm Ext}_{R}^1(M,F)$ and by \cite[Theorem 3.8.5]{KIMBOOK}, $F$ is a flat $R$-module and so 	${\rm Ext}_{R}^1(M,F)$ is a GV-torsion $R$-module. Hence, ${\rm Ext}_{T}^1(T \otimes_R M,F)$ is a GV-torsion $T$-module by Definition \ref{def bouba}. Hence, $T \otimes_R M$ is a $w$-copure projective $T$-module.
\end{proof}

%\begin{corollary}\label{coro2}
%	Let $\phi:R\rightarrow T$ be a ring homomorphism, $T$ be a projective $R$-moduleand, and $T$ has property ($B_\phi$). If $M$ is a $w$-copure projective $R$-module, then $T\otimes_R M$ is a $w$-copure projective $T$-module.
%\end{corollary}

\begin{corollary}\label{coro3}
	Let $M$ be a strongly $w$-copure projective $R$-module.
	\begin{enumerate}
		\item  $M[X]$ is a $w$-copure projective $R[X]$-module.
		\item  If $S$ is a multiplicative subset of $R$, then $M_S$ is a $w$-copure projective $R_S$-module.
	\end{enumerate}
\end{corollary}
\begin{proof}
 $(1)$ and $(2)$ follow by setting $T:=R[X]$ and $T:=R_S$ respectively in
 Proposition \ref{propo2}.
\end{proof}

\begin{proposition}\label{propo3}
	Let $M$ be a strongly $w$-copure projective $R$-module. Then  $M\{X\}$ is a strongly $w$-copure projective $R\{X\}$-module.
\end{proposition}
\begin{proof}
	Let $M$ be a strongly $w$-copure projective $R$-module. So, by Corollary \ref{coro3} $(1)$, $M[X]$ is $w$-copure projective. By Corollary \ref{coro3} $(2)$, $M\{X\}$ is a  $w$-copure projective $R\{X\}$-module. Since $R\{X\}$ is a $DW$-ring by \cite[Theorem 6.6.18]{KIMBOOK}, $M\{X\}$ is a  strongly $w$-copure projective $R\{X\}$-module.
\end{proof}	
\begin{proposition}\label{carw1}
	Let $R$ be a coherent ring. Then:
	\begin{enumerate}
		\item  If $M$ is a strongly $w$-copure projective $R$-module, then $M\{X\}$ is a strongly $w$-copure flat $R\{X\}$-module.
		\item  If $M$ is a finitely presented strongly $w$-copure flat $R$-module, then $M\{X\}$ is a strongly $w$-copure projective $R\{X\}$-module.
	\end{enumerate}
\end{proposition}
\begin{proof}
	$(1)$ Let $M$ be a strongly $w$-copure projective $R$-module, then, by Proposition \ref{carw}, $M$ is strongly $w$-copure flat. So, by \cite[Proposition 4.6]{BKM} we have $M\{X\}$ is a strongly $w$-copure flat $R\{X\}$-module.
	
	$(2)$ Let $M$ be a finitely presented strongly $w$-copure flat $R$-module. Then, by Proposition \ref{carw}, $M$ is strongly $w$-copure projective and  by Proposition \ref{propo3}, $M\{X\}$ is a strongly $w$-copure projective $R\{X\}$-module.
\end{proof}

\end{document}